\documentclass[11pt,reqno]{amsart}
\usepackage[margin=1.25in, footskip=.5in]{geometry}
\usepackage[T1]{fontenc}
\usepackage[american]{babel}
\usepackage{amsmath,amsfonts,amssymb,amsthm,bm,wasysym}
\usepackage{mathrsfs}
\usepackage{graphicx}
\usepackage{enumitem}
\usepackage{multirow}
\usepackage{caption,float}
\usepackage{subcaption}
\usepackage{tikz}
\usetikzlibrary{matrix,arrows,backgrounds,calc,chains,automata,positioning,patterns}
\usetikzlibrary{decorations,decorations.pathmorphing}
\usetikzlibrary{shapes.geometric,calc}
\usepackage[colorlinks,linkcolor=green!50!black,citecolor=blue!50!black,pagebackref,hypertexnames=false, breaklinks]{hyperref}
\allowdisplaybreaks%to break long aling with newpage
\newtheorem{theorem}{Theorem}%[section]
\newtheorem{proposition}[theorem]{Proposition}
\newtheorem{corollary}[theorem]{Corollary}
\newtheorem{lemma}[theorem]{Lemma}
\theoremstyle{definition}
\newtheorem{definition}{Definition}
\newtheorem*{def*}{Definition}

\newtheorem{remark}{Remark}

\newtheorem*{thm*}{Theorem}
\newtheorem*{lem*}{Lemma}
\newtheorem*{prop*}{Proposition}
\newtheorem*{rem*}{Remark}

\def\supp{\operatorname{supp}}

\def\SG{{\rm SG}}
\newcommand{\N}{\mathbb{N}}
\newcommand{\R}{\mathbb{R}}

\newcommand{\T}{\mathbb{T}}
\numberwithin{equation}{section}
\title[NLS on the Sierpinski gasket]{Iterative methods fail to solve NLS\\ below the Sobolev embedding threshold \\ on the Sierpinski gasket}
\date{\today}
\author{Patricia Alonso Ruiz}
\address{Institute of Mathematics, Friedrich Schiller University Jena, Germany}
\email{patricia.alonso.ruiz@uni-jena.de}
\author{Gigliola Staffilani}
\thanks{P.A.R. was partly funded by the NSF CAREER grant DMS-2140664. G.S. was funded in part by NSF DMS-2052651, 
DMS-2306378 and the Simons Foundation
Collaboration Grant on Wave Turbulence.}
\address{Department of Mathematics, MIT, USA}
\email{gigliola@math.mit.edu}

\begin{document}
\begin{abstract}
We show that the nonlinear Schr\"odinger equation on the Sierpinski gasket with a power nonlinearity of order $2k{+}1$ is \emph{not} locally well-posed for initial data \emph{just below} the regularity threshold for the Sobolev embedding $H^s\subseteq L^\infty$. More precisely, the flow map fails to be $C^{2k+1}$-continuous in any Sobolev space $H^s$ below that threshold, and the threshold is independent of the power nonlinearity. This novel behavior significantly differs from other compact spaces such as the torus or the sphere, and it is directly connected to the existence of localized eigenfunctions.
\end{abstract}
\maketitle

{\small

}

\tableofcontents
\newpage
%%%%%%%---------------------------------------------------
%%%%%%%---------------------------------------------------
\section{Introduction}
In this paper we investigate the well-posedness of the Cauchy problem associated to a nonlinear Schr\"odinger equation (NLS) on a compact domain of fractal type, namely the Sierpinski gasket (SG), see Figure~\ref{F:SG}. Before moving to the remarkable properties of this domain and their interactions with the analysis of the Cauchy problem, we first provide the definition of well-posedness adopted throughout the discussion.

\begin{definition}[Strong well-posedness]\label{D:well-p} We say that the Cauchy problem associated to an evolution equation is \emph{strongly well-posed} in the Sobolev space $H^s$ if, for any initial data $u_0\in H^s$, there exists a time $T>0$ depending only on $u_0$, and a subspace $X_T^s\subset C([-T,T], H^s)$, such that
\begin{enumerate}[wide=0em,label=(\roman*)]
    \item there exists a solution  $u(t,x)\in X_T^s$ with $u(0,\cdot)=u_0$,  
    \item the solution is unique in $X_T^s$, 
    \item the solution map that associates $u_0$ to $u(t,x)$ as a map from $H^s$ to $X_T^s$ is smooth in the topology of $C([-T,T];H^s)$.
\end{enumerate}
\end{definition}

In particular, we focus on the Cauchy problem for the NLS
\begin{equation}\label{E:NLS_intro}
        \begin{cases}
        i\partial_t u+\Delta u= \mu |u|^{2k}u&\quad (t,x)\in \mathbb{R}\times M,\\
        u(0,\cdot)=u_0\in H^s_x(M),&
        \end{cases}
    \end{equation}
where $k\in \N$, $\mu=\pm 1$, and $M$ may be the Euclidean space $\R^d$ or a compact space, as will be made clearer below. 
When $M=\R^d$, we recall that NLS enjoys the following scaling invariance: if $u(t,x)$ is a solution to \eqref{E:NLS_intro}, then 
\begin{equation}\label{E:scaling_NLS}
    u_\lambda(t,x):= \lambda^{-\frac{1}{k}}u(t/\lambda^2, x/\lambda),\quad \lambda>0,
 \end{equation}
also solves~\eqref{E:NLS_intro} with initial data $u_{0,\lambda}:=\lambda^{-\frac{1}{k}}u_0( x/\lambda)$. In addition, the homogenous Sobolev norm $\dot H^{s_c}$ is invariant under the scaling~\eqref{E:scaling_NLS} for
\begin{equation}\label{E:critical} 
  s_c:= \frac{d}{2} - \frac{1}{k},
\end{equation} 
see e.g.~\cite[Section 3.1]{Tao06}. A relevant observation here is that, in the case $M=\R^d$ or $M=\T^d$, 
\[
s_c< \sigma_\infty,
\] 
where $\sigma_\infty=\frac{d}{2}$ is the threshold for the Sobolev embedding $H^{s}\subseteq L^\infty$. As a consequence of the embedding, $H^{s}$ is an algebra for $s>\sigma_\infty$, and it is then standard to show via a contraction  method that~\eqref{E:NLS_intro} is strongly well-posed in $H^s$ with $s>\sigma_\infty$. 

\medskip

When working in low regularity regimes it is customary to say that the problem~\eqref{E:NLS_intro} has a solution $u(t,x)$ in a time interval $[0,T]$, if the latter satisfies the Duhamel's formula
\begin{equation}\label{E:Duhamel}
    u(t,x)=e^{it\Delta}u_0(x)+i\mu \int_0^t e^{i(t-\tau)\Delta}|u|^{2k}u(\tau,x)\,d\tau
\end{equation}
for all $t\in[0,T]$, $\mu=\pm 1$. 
Based on our definition of strong well-posedness, the Cauchy problem \eqref{E:NLS_intro}, and now~\eqref{E:Duhamel}, is generally considered \emph{ill-posed} as soon as any of the requirements (i)-(iii) from Definition~\ref{D:well-p} fail. In that regard, one may speak of a \emph{strong well-posedness threshold} for~\eqref{E:NLS_intro} as the number $s_w$ such that
\begin{equation*}
    \text{NLS}~\eqref{E:NLS_intro}\text{ is }
    \begin{cases}
     \text{strongly well-posed in } H^s&\text{for }s> s_w,\\
    \text{ ill-posed in } H^s&\text{for }s<s_w.
    \end{cases}
\end{equation*}
Investigating mild forms of ill-posedness can serve to find (at least a proxy for) $s_w$. For instance, one can analyze at which regularity the flow map $u_0(x)\mapsto u(t,x)$ fails to act smoothly on $H^s$, since in that case no iterative scheme for~\eqref{E:Duhamel} will be able to provide (at least) existence of solutions to~\eqref{E:NLS_intro}, c.f.~\cite[Section 6]{Bou97}. This kind of ill-posedness is sometimes referred to as $C^\gamma$-ill-posedness when the solution map fails to be $C^\gamma$, see e.g.~\cite[Section 3.8]{Tao06}.

\medskip

The Cauchy problem for the cubic NLS on the one-dimensional torus $\mathbb{T}$ was proved to be ill-posed for initial data $u_0\in H^s(\mathbb{T})$ and $s<0$ by Molinet in~\cite{Mol09}, where he showed that the flow-map is not continuous from $H^s(\T)$ into itself. Strong local well-posedness when $s\geq 0$ had been proved by Bourgain in~\cite{Bou93} using a fixed point argument, whence in this case $s_w=0$ for $k=1$. Noteworthy is that $s_w>s_c=-\frac{1}{2}$ in this case, c.f.~\eqref{E:critical}.
Ill-posedness of NLS on $\mathbb{T}$ with a power nonlinearity of order $2k+1$ for $u_0\in H^s(\mathbb{T})$ was addressed by Burq-G\'erard-Tzvetkov in~\cite{BGT02}, where they proved that the flow map fails to be uniformly continuous when $s<0$, c.f.~\cite[Remark 1.2]{BGT02}. In fact, Chirst-Colliander-Tao showed in~\cite[Theorem 1]{CCT03b} that, when $u_0\in H^s(\mathbb{T})$ and $s<0$, the solution map fails to be continuous from $H^s(\mathbb{T})$ to any $H^{\sigma}(\mathbb{T})$ with $\sigma<0$. On the other hand, using Strichartz estimates and a fixed point argument one can prove strong local well-posedness for $H^s(\mathbb{T})$ for $s>s_c= \frac{1}{2}-\frac{1}{k}$ and $k\geq 2$, see~\cite[Theorem 1]{Bou93}, and~\cite[Theorem 1.1]{KK24} for recent results at the critical regularity $s=s_c$ when $k>2$. 
 In particular, this shows that $s_w=0=s_c$ for $k=2$. Thus far, ill-posedness below $s_c$ seems to be an open question when $k>2$, and therefore we can only say $s_w\in [0,s_c]$.

\medskip

Moving to two-dimensional compact model spaces, local well-posedness fails for the cubic NLS on $\mathbb{T}^2$ with $u_0\in H^s$ for $s<0$, see~\cite[Section 11]{CCT03}. Thanks to Bourgain's work~\cite[Proposition 5.73]{Bou93}, see also~\cite[Theorem 40]{Bou99} and the recent contribution~\cite{HK24}, strong well-posedness holds for initial data in $H^s(\mathbb{T}^2)$ and $s>0$. Thus, $s_w=0$ for $k=1$, and $s_w=0=s_c$, as is the case for the quintic ($k=2$) NLS on the one-dimensional torus. Strong well-posedness for the cubic NLS with initial data in $L^2(\mathbb{T}^2)$ remains an open problem. When $k\geq 2$, local well-posedness for $s\geq s_c=\frac{1}{2}-\frac{1}{k}$ follows from~\cite[Theorem 1.1]{KK24} by propagation of smoothness. To the best of our knowledge, ill-posedness below $s_c$ when $k\geq 2$ has not been investigated, so that $s_w\in[0,s_c]$ in this case.

\medskip

The picture in the case of the two-dimensional sphere is different: Burq-G\'erard-Tzvetkov proved well-posedness of the cubic NLS in $H^s(\mathbb{S}^2)$ for $s>\frac{1}{4}$ in~\cite[Theorem 1]{BGT05} and ill-posedness for $s<\frac{1}{4}$ in~\cite[Theorem 2, Remark 1.3]{BGT02}. See also~\cite{Ban04} and the recent contribution~\cite{B+24} for more results in this direction. Hence, $s_w=\frac{1}{4}$ for $k=1$. When $k\geq 2$, the aforementioned authors showed well-posedness in $H^s(\mathbb{S}^2)$ for $s>1-\frac{1}{2k}>s_c$, see~\cite[Proposition 3.1]{BGT04}. To the best of our knowledge, ill-posedness below $1-\frac{1}{2k}$ when $k\geq 2$ has not been investigated yet. Therefore, $s_w\leq 1-\frac{1}{2k}$ in this case.

\medskip

Thus far, as remarked above, the literature seems to have left open the question of the failure of local well-posedness for NLS with a generic $2k+1$ power nonlinearity on $\mathbb{T}^2$ and $\mathbb{S}^2$. It is however expected that the strong well-posedness threshold $s_w$ depends on the order of the power nonlinearity. 
Thanks to Strichartz estimates and a fixed point argument in these models, we do know that $s_w$ always lies ``significantly below'' the regularity threshold $\sigma_\infty$ for which the Sobolev embedding $H^{s}\subseteq L^\infty$ holds with $s>\sigma_\infty$.

\medskip

The above discussion should highlight the new behavior observed when the underlying space $M$, while remaining compact, is fractal. In the present paper we analyze in detail the case of the Sierpinski gasket (SG), taken as prototype for being one of the most studied compact fractal sets, whose (Hausdorff) dimension precisely lies between that of $\mathbb{T}$, and $\mathbb{T}^2$ or $\mathbb{S}^2$. 

\medskip

The study of the standard Laplace operator associated with SG and its spectral properties has a long history, going back to the physicists Rammal and Tolouse~\cite{RT82,Ram84} and the subsequent work by Fukushima and Shima in~\cite{FS92}. Section~\ref{SS:SG_basics} briefly recalls the functional analytic tools required to rigorously formulate~\eqref{E:NLS_intro} in this setting, and we suggest the interested reader to delve into the topic through the books by Kigami~\cite{Kig01} and Strichartz~\cite{Str06}. 

\medskip

The {first} main result in this paper shows that NLS with a $2k+1$ power nonlinearity is $C^{2k+1}$-ill-posed \emph{for any} regularity $s$ \emph{below} the threshold $\sigma_\infty$ for the Sobolev embedding $H^s\subseteq L^\infty$, \emph{and for any} $k$ in the nonlinearity. 

\begin{theorem}\label{T:no_cont_below_Sobolev}
    Let $0<s<\sigma_\infty(\SG)$. For $\mu=\pm 1$, for every integer $k\geq 1$, the non-linear Schr\"odinger equation
    \begin{equation*}%\label{E:NLS_intro_SG}
        \begin{cases}
        i\partial_t u+\Delta_{\SG} u= \mu|u|^{2k}u&\quad (t,x)\in \mathbb{R}\times\SG,\\
        u(0,\cdot)=u_0\in H^s(\SG),&
        \end{cases}
    \end{equation*}
    is $C^{2k+1}$-ill-posed. {In particular, the flow-map cannot extend from $H^1(\SG)$ to $H^s(\SG)$ with $1<s<\sigma_\infty$ as a $C^{2k+1}$-flow-map near the origin.}
\end{theorem}

\begin{remark}\label{R:strong_LWP}
    In view of the Sobolev embedding, see Theorem~\ref{T:Sobolev_embedding} below, a standard fixed-point argument leads to strong local well-posedness in $H^s(\SG)$ for {$s=1>\sigma_\infty$ since $H^1(\SG)$ is an algebra by virtue of the general abstract estimate~\cite[(3.2.28)]{FOT11}. Note however that, in contrast to more classical models, the space $H^s(\SG)$ may \emph{not} be an algebra for $s>\sigma_\infty$: Currently it is known that $H^2(\SG)$ fails to be an algebra, see e.g.~\cite{BST99}, while the general picture is much less clear and its investigation subject of future research.}
\end{remark}

{A} novelty arising from Theorem \ref{T:no_cont_below_Sobolev} is that, on SG, the threshold $s_w=\sigma_\infty$ is relatively ``high'' in comparison to the known/expected values in the model spaces $\mathbb{T}$, $\mathbb{T}^2$, $\mathbb{S}^2$, where Strichartz estimates make it possible to prove well-posedness below the  $\sigma_\infty$ threshold.
On the other hand the phenomenon of ill-posedness below the Sobolev embedding threshold has been also observed in the cubic Szeg\H{o} equation on sub-Riemannian manifolds in~\cite{GG10}. 

{Another interesting} observation is that for SG the threshold $s_w$ is always \emph{independent} of the $2k+1$ power nonlinearity in NLS. As explained in detail in Subsection~\ref{SS:Eigenfunctions}, these new phenomena seem to be tightly related to the existence of so-called localized eigenfunctions of the Laplacian on SG. 

\medskip

For the convenience of the reader, the results mentioned in the introduction and the new result in this paper are summarized in Table~\ref{Table:ill_posed_thresholds}.

\begin{table}[H]
    \centering
    \setlength{\tabcolsep}{.5em}
    \renewcommand{\arraystretch}{2}
    \begin{tabular}{c|c|c|c|c}
    $|u|^{2k}u$ & $\mathbb{T}$  & $\SG$ & $\mathbb{T}^2$ & $\mathbb{S}^2$\\ \hline
    $k=1$ & $0=s_w{\color{red}<}\,\sigma_\infty=\frac{1}{2}$  & $s_w\,{\color{red}=}\,\sigma_\infty\,{=\frac{\log 3}{\log 5}}$ &  $0=s_w{\color{red}<}\,\sigma_\infty=1$& $\frac{1}{4}=s_w{\color{red}<}\,\sigma_\infty=1$\\
    $k= 2$ & $0=s_w\,{\color{red}<}\,\sigma_\infty=\frac{1}{2}$  & $s_w\,{\color{red}=}\,\sigma_\infty\,{=\frac{\log 3}{\log 5}}$ &  $0=s_w\,{\color{red}< }\,\sigma_\infty=1$ & $?=s_w\,{\color{red}{<}}\,\sigma_\infty=1$\\
    $k\geq 3$ & $?=s_w\,{\color{red}<}\,\sigma_\infty=\frac{1}{2}$  & $s_w\,{\color{red}=}\,\sigma_\infty\,{=\frac{\log 3}{\log 5}}$ &  $?=s_w\,{\color{red}{<}}\sigma_\infty=1$ & $?=s_w\,{\color{red}{<}}\sigma_\infty=1$\\
\end{tabular}
\caption{Summary of results on compact models ``between'' the Sierpinski gasket.}
\label{Table:ill_posed_thresholds}
\end{table}

%%%%-------------- After the referee's question to clarify ----------------
{
\begin{remark}
    A stronger notion of ill-posedness is non-uniform continuity of the flow map, such as for example the result proved for $\mathbb{S}^2$ in~\cite{BGT02}. The approach there profited from the possibility to write an explicit solution to~\ref{E:NLS_intro}. Finding an explicit (or closed) expression for the localized eigenfunctions on the Sierpinski gasket involved in our analysis, c.f. Proposition~\ref{P:existence_localized}, would open the possibility to apply this approach. We conjecture that also in this case the flow map will exhibit a non-uniform continuity.
\end{remark}
}

{The second main result in this work  concerns the failure in SG of  Strichartz estimates  below the trivial threshold given by the Sobolev embedding.  If $M$ is a compact manifold of dimension $d\leq 2$, for the case of the cubic NLS, that is $k=1$ in~\eqref{E:NLS_intro},  to address well-posedness one uses the estimate 
\begin{equation}\label{Stri}
\|S_\tau u\|_{L^4_tL^4_x([0,T]\times M)}\apprle \|u\|_{H^s(M)},
\end{equation}
that  is always proved for   $s<\frac{d}{4}$, see  \cite{BGT05}, where $ \frac{d}{4}$ is the regularity needed to control  the $L^4$ norm in $M$ via Sobolev embeddijng. In Subsection  \ref{SS:No_Strichartz_cubic} we show that in SG   estimate \eqref{Stri} cannot hold for $s<\frac{d}{4}$. More precisely we have the following theorem.
\begin{theorem}\label{T:no_Strichartz_cubic}
For any $T>0$ there exists $\psi\in H^{d_S/4}$ such that
    \begin{equation}\label{E:no_Strichartz_cubic}
        \|S_t\psi\|_{L^4_tL^4_x([0,T]\times\SG)}\simeq_{d_S}T\|\psi\|_{H^{d_S/4}}.
    \end{equation}
\end{theorem}
\begin{remark}
It is possible to obtain a similar result as Theorem \ref{T:no_Strichartz_cubic} for $\|S_t\psi\|_{L^{2k+2}_tL^{2k+2}_x}$ and an appropriate Sobolev norm on the right-hand side. For conciseness, the details are left to the interested reader. 
It would also be interesting to investigate local well-posedness for~\eqref{E:NLS_intro} with $k=1$ in $H^s$ with $\frac{d_S}{4}<s<\sigma_\infty$.
\end{remark}
}

\smallskip
The paper is organized as follows: Section~\ref{S:Background} provides a brief review of the Sierpinski gasket, its standard Laplacian, and the functional spaces required to rigorously state and analyze NLS. An essential role is played by certain eigenfunctions of the Laplacian described in Section~\ref{S:Sobolev_embedding} along with the properties that render them key to prove the main result, which is proved in Section~\ref{S:main_result}.

\section*{Acknowledgments}
The authors are greatly thankful to K. Okjoudou and M. Taylor for inspiring discussions that originated this project, for the hospitality of ITS-ETH, where part of this work was conducted, and to J. L\"uhrmann for comments on earlier versions of the manuscript. {They also leave a special note of thanks to the anonymous referee, whose report significantly contributed to improving the original paper and drawn their attention to the  question of the validity of Strichartz estimates below the regularity dictated by the Sobolev embedding.}

%%%%%%%---------------------------------------------------
\section{Background and preliminaries}\label{S:Background}
%%%%%%%---------------------------------------------------
%%%%%%%---------------------------------------------------
\subsection{The Sierpinski gasket}\label{SS:SG_basics}
%%%%%%%---------------------------------------------------
The standard Sierpinski gasket ($\SG$), see Figure~\ref{F:SG}, is the unique compact subset of $\mathbb{R}^2$ that satisfies the equation
\begin{equation}\label{E:SG_fixed_point}
    \SG=F_0(\SG)\cup F_1(\SG)\cup F_2(\SG),
\end{equation} 
where $F_i\colon\mathbb{R}^2\to\mathbb{R}^2$, $i=0,1,2$, are the mappings
\begin{equation*}
    F_i(p):=\frac{1}{2}(p-p_i)+p_i
\end{equation*}
and $V_0:=\{p_0,p_1,p_2\}$ denotes the set of vertices of an equilateral triangle of side length one, c.f. Figure~\ref{F:SG_approx}. The set $V_0$ is also regarded as the natural boundary of $\SG$. The self-similarity property~\eqref{E:SG_fixed_point} has its roots in the \emph{Banach fixed point theorem}, c.f.~\cite[Section 3.1]{Hut81}.
As a metric measure space, $\SG$ is usually equipped with the Euclidean metric and the standard self-similar Bernoulli measure $\mu$ that gives the same weight to each triangular cell of the same side-length. This measure thus satisfies
\begin{equation}\label{E:def_measure_SG}
    \mu(F_{w}(V_0))=\frac{1}{3^m}
\end{equation}
for any $m\geq 0$ and $F_{w}:=F_{w_1}{\circ}\cdots \circ F_{w_m}$ for $w=w_1\ldots w_m\in\{0,1,2\}^m$. The latter is comparable to the
normalized $d_H$-dimensional Hausdorff measure, where $d_H=\frac{\log 3}{\log 2}$ is the (Euclidean) Hausdorff dimension of $\SG$.

\begin{figure}[H]
    \includegraphics[scale=.15]{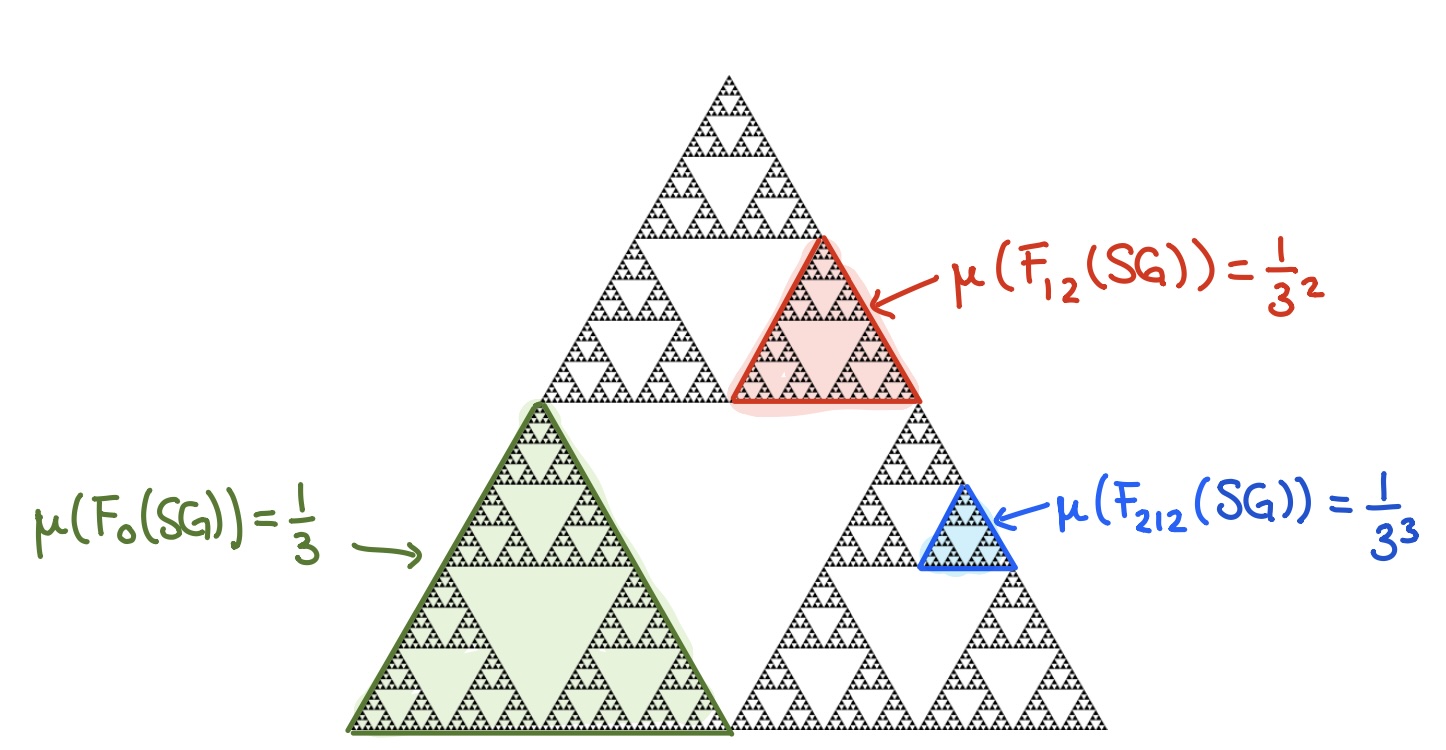}
    \caption{The standard Sierpinski gasket and the measure of different cells.}
    \label{F:SG}
\end{figure}

A common way to construct an intrinsic Laplace operator on $\SG$ to act as the usual second derivative in Euclidean space is through finite graph approximations. We review here the basics and refer the interested reader to the books~\cite{Kig01,Str06} for further details.
For each $m\geq 0$, let $V_m$ denote the vertex set of the finite $m$-level approximation of the Sierpinski gasket, see Figure~\ref{F:SG_approx}. 
%%%----------------------------------------------
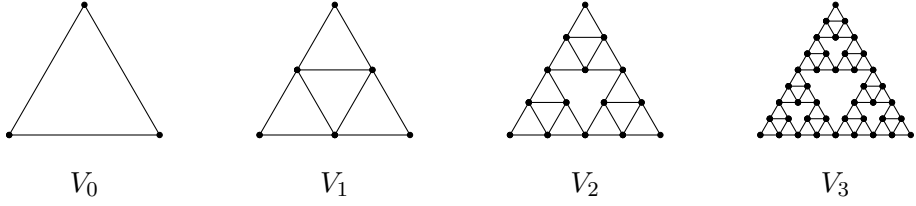
\begin{figure}[H]
\begin{center}
\renewcommand{\arraystretch}{0.5}
\begin{tabular}{cccc}
\begin{tikzpicture}
\tikzstyle{every node}=[draw,circle,fill=black,minimum size=2pt, inner sep=0pt]
\draw ($(0:0)$) node () {} --++ ($(0:2)$) node () {} --++ ($(120:2)$) node () {} --++ ($(240:2)$) node () {};
\end{tikzpicture}
\hspace*{2em}&
\begin{tikzpicture}
\tikzstyle{every node}=[draw,circle,fill=black,minimum size=2pt, inner sep=0pt]
\draw ($(0:0)$) node () {} --++ ($(0:2/2)$) node () {} --++ ($(120:2/2)$) node () {} --++ ($(240:2/2)$) node () {};
\draw ($(0:2/2)$) node () {} --++ ($(0:2/2)$) node () {} --++ ($(120:2/2)$) node () {} --++ ($(240:2/2)$) node () {};
\draw ($(60:2/2)$) node () {} --++ ($(0:2/2)$) node () {} --++ ($(120:2/2)$) node () {} --++ ($(240:2/2)$) node () {};
\end{tikzpicture}
\hspace*{2em}&
\begin{tikzpicture}
\tikzstyle{every node}=[draw,circle,fill=black,minimum size=2pt, inner sep=0pt]
\draw ($(0:0)$) node () {} --++ ($(0:2/4)$) node () {} --++ ($(120:2/4)$) node () {} --++ ($(240:2/4)$) node () {};
\draw ($(0:2/4)$) node () {} --++ ($(0:2/4)$) node () {} --++ ($(120:2/4)$) node () {} --++ ($(240:2/4)$) node () {};
\draw ($(60:2/4)$) node () {} --++ ($(0:2/4)$) node () {} --++ ($(120:2/4)$) node () {} --++ ($(240:2/4)$) node () {};
\foreach \a in {0,60}{
\draw ($(\a:2/2)$) node () {} --++ ($(0:2/4)$) node () {} --++ ($(120:2/4)$) node () {} --++ ($(240:2/4)$) node () {};
\foreach \b in{0,60}{
\draw ($(\a:2/2)+(\b:2/4)$)node () {} --++ ($(0:2/4)$) node () {} --++ ($(120:2/4)$) node () {} --++ ($(240:2/4)$) node () {};
}
}
\end{tikzpicture}
\hspace*{2em}&
\begin{tikzpicture}
\tikzstyle{every node}=[draw,circle,fill=black,minimum size=2pt, inner sep=0pt]
%%%-------- lower-left corner ----------------
\draw ($(0:0)$) node () {} --++ ($(0:2/8)$) node () {} --++ ($(120:2/8)$) node () {} --++ ($(240:2/8)$) node () {};
\draw ($(0:2/8)$) node () {} --++ ($(0:2/8)$) node () {} --++ ($(120:2/8)$) node () {} --++ ($(240:2/8)$) node () {};
\draw ($(60:2/8)$) node () {} --++ ($(0:2/8)$) node () {} --++ ($(120:2/8)$) node () {} --++ ($(240:2/8)$) node () {};
\foreach \a in {0,60}{
\draw ($(\a:2/4)$) node () {} --++ ($(0:2/8)$) node () {} --++ ($(120:2/8)$) node () {} --++ ($(240:2/8)$) node () {};
\foreach \b in{0,60}{
\draw ($(\a:2/4)+(\b:2/8)$)node () {} --++ ($(0:2/8)$) node () {} --++ ($(120:2/8)$) node () {} --++ ($(240:2/8)$) node () {};
}
}
\foreach \c in{0,60}{
\draw ($(\c:2/2)$) node () {} --++ ($(0:2/8)$) node () {} --++ ($(120:2/8)$) node () {} --++ ($(240:2/8)$) node () {};
\draw ($(\c:2/2)+(0:2/8)$) node () {} --++ ($(0:2/8)$) node () {} --++ ($(120:2/8)$) node () {} --++ ($(240:2/8)$) node () {};
\draw ($(\c:2/2)+(60:2/8)$) node () {} --++ ($(0:2/8)$) node () {} --++ ($(120:2/8)$) node () {} --++ ($(240:2/8)$) node () {};
\foreach \a in {0,60}{
\draw ($(\c:2/2)+(\a:2/4)$) node () {} --++ ($(0:2/8)$) node () {} --++ ($(120:2/8)$) node () {} --++ ($(240:2/8)$) node () {};
\foreach \b in{0,60}{
\draw ($(\c:2/2)+(\a:2/4)+(\b:2/8)$)node () {} --++ ($(0:2/8)$) node () {} --++ ($(120:2/8)$) node () {} --++ ($(240:2/8)$) node () {};
}
}
}

%%%----- copies up & lower-right

\end{tikzpicture}
\\ [1em]
$V_0$ \hspace*{2em}& $V_1$ \hspace*{2em}& $V_2$ \hspace*{2em}& $V_3$
\end{tabular}
\end{center}
\caption{Graph approximations of the Sierpinski gasket.}\label{F:SG_approx}
\end{figure}

The graph energy associated with the $m$-level approximation is given by
\begin{equation}\label{E:graph_energy_m}
\mathcal{E}_m (u):=\frac{1}{2}\sum_{p\stackrel{m}{\sim} q}(u(q)-u(p))^2,
\end{equation}
where $p\stackrel{m}{\sim}q$ means that $p,q\in V_m$ are neighbors. The standard energy on $\SG$ arises as the limit
\begin{equation*}
    \mathcal{E}(u):=\lim_{m\to\infty}\Big(\frac{5}{3}\Big)^m\mathcal{E}_m (u),
\end{equation*}
which is extended to a bilinear form $\mathcal{E}(u,v)$ through polarization. The ``Neumann domain'' of $\mathcal{E}$, denoted by $\mathcal{F}$, is obtained as the closure of continuous functions with respect to $(\mathcal{E}(u)+\|u\|_{L^2})^{1/2}$. The Dirichlet domain $\mathcal{F}_0$, consists of those functions in $\mathcal{F}$ that are zero on the boundary $V_0$. The bilinear forms $(\mathcal{E},\mathcal{F})$ and $(\mathcal{E},\mathcal{F}_0)$ are in fact local and regular Dirichlet forms~\cite[Theorem 3.4.6]{Kig01}.

\medskip

As such, each Dirichlet form has an associated infinitesimal generator, $-\Delta_N$ and $-\Delta_D$, that is a non-negative self-adjoint operator corresponding to the Laplace operator on $\SG$ with Neumann, respectively Dirichlet, boundary conditions. More precisely,
\begin{enumerate}[wide=0em,label=(\roman*),itemsep=.5em]
    \item for any $u\in\mathcal{F}$, we say that $-\Delta_Nu=f$ if and only if
    \begin{equation*}%\label{E:def_Neumann_Laplacian}
        \mathcal{E}(u,v)=\langle f,v\rangle\qquad\forall\,v\in\mathcal{F},
    \end{equation*}
    \item for any $u\in\mathcal{F}_0$, we say that $-\Delta_Du=f$ if and only if
    \begin{equation*}%\label{E:def_Dirichlet_Laplacian}
        \mathcal{E}(u,v)=\langle f,v\rangle\qquad\forall\,v\in\mathcal{F}_0.
    \end{equation*}
\end{enumerate}
Throughout the exposition, we will drop the index N and D from $-\Delta$ whenever a statement holds equally for both operators.  
Since $-\Delta$ is self-adjoint and $L^2(\SG,\mu)$ is Hilbert, the operator admits a unique spectral resolution, that is
\begin{equation}\label{E:spectral_resolution}
    -\Delta u=\int_0^\infty \lambda\,dE_\lambda(u)
\end{equation}
for all $u\in{\rm dom}\,(-\Delta)$, where $\lambda$ denotes the spectral parameter and $dE_\lambda$ the $L^2$-valued spectral measure associated with $-\Delta$, 
see e.g.~\cite[Theorem 1, Section XI.6]{Yos80}. 

\medskip

As the infinitesimal generator of a Dirichlet form, the Laplace operator $-\Delta$ is associated with a semigroup $\{P_t\}_{t\geq 0}$ that solves the heat equation in the sense that 
\[
\partial_s P_su=-\Delta P_su
\]
for any $u\in {\rm dom}\,(-\Delta)$, c.f.~\cite[(A.1.5)]{BGL14}. In view of~\eqref{E:spectral_resolution} and by virtue of the spectral theorem, see e.g.~\cite[Theorem A.4.2]{BGL14}, the semigroup also admits the expression
\begin{equation}\label{E:def_Pt}
    P_tu=\int_0^\infty e^{-\lambda t}dE_\lambda(u)
\end{equation}
for any $u\in L^2(\SG,\mu)$. In addition, it has an associated heat kernel $p_t(x,y)$ that satisfies sub-Gaussian estimates of the type
\begin{equation}\label{E:subG_HKE}
    p_t(x,y)
    \asymp c_1 t^{-d_S/2} \exp \bigg( - c_2\Big(\frac{d(x,y)^{d_W}}{t}\Big)^{\frac1{d_W-1}} \bigg)
            \end{equation}
for $\mu$-a.e. $x,y\in\SG$ and $t\in (0,t_0)$, see e.g.~\cite{Kum93}. Here, $d_S$ and $d_W$ denote the spectral, respectively the walk, dimension of $\SG$. We refer to~\cite{Kum93} and~\cite[Section 5]{Bar98} for details about these dimensions, which in the case of the Sierpinski gasket are
\begin{equation*}%\label{E:dimensions_SG}
    d_W=\frac{\log 5}{\log 2}\qquad\text{and}\qquad d_S=\frac{\log 9}{\log 5}=\frac{2d_H}{d_W}.
\end{equation*}
 Note in particular, that $d_W>2$ and thus
 \begin{equation*}%\label{E:dS_below_one}
    \frac{d_S}{2}=\frac{d_H}{d_W}<\frac{d_H}{2}<1.
    \end{equation*}   
As it will be explained in the next section, the latter quantity will be the regularity threshold in the main result, Theorem~\ref{T:no_cont_below_Sobolev}.

%%%%%%%---------------------------------------------------
\subsection{Fractional Sobolev spaces}
%%%%%%%--------------------------------------------------- 
We recall the construction of fractional Sobolev spaces introduced by Strichartz in~\cite{Str03}. For any $s>0$, the spectral theorem allows one to write 
\begin{equation*}%\label{E:def_Bessel}
    (I-\Delta)^{-s/2}u=\int_0^\infty (1+\lambda)^{-s/2}dE_\lambda(u),
\end{equation*}
for $u\in L^2(\SG,\mu)$, see e.g.~\cite[Theorem A.4.2]{BGL14}. Now, consider the image of $L^2(\SG,\mu)$ under that operator, i.e.
\begin{equation*}%\label{E:def_Strichartz_L2s}
    H^s_{N}(\SG):=\big\{(I-\Delta_N)^{-s/2}u\colon u\in L^2(\SG,\mu)\big\}
\end{equation*}
and define analogously $H^s_D(\SG)$ with the Dirichlet Laplacian $-\Delta_D$. These spaces were studied from a potential-theoretic viewpoint in~\cite{HZ05}. 

\medskip

Further, consider the space of harmonic functions
\begin{equation*}%\label{E:def_D_harmonics}
    \mathcal{H}_0(\SG):=\{u\in \mathcal{F}\colon~\mathcal{E}(u,v)=0~\forall\,v\in\mathcal{F}_0\}
\end{equation*}
c.f.~\cite[Section 3]{Str03} or~\cite[Section 3]{CQ22}.

\begin{definition}%\label{D:Hs_SG}
    Let $0<s<2$. The (Neumann, respectively Dirichlet) fractional Sobolev space of order $s$ is defined as
    \begin{equation}\label{E:def_Hs}
        \begin{aligned}
        H^s(\SG):=H^s_{N}(\SG)\oplus \mathcal{H}_{0}(\SG),\\
        H^s_0(\SG):=H^s_{D}(\SG)\oplus \mathcal{H}_{0}(\SG),
        \end{aligned}
    \end{equation}
    and the associated seminorms are given by
    \begin{equation}\label{E:def_seminorm_Hs}
        \|u\|_{\dot{H}^s}:=\|(I-\Delta_N)^{s/2}u\|_{L^2}\qquad\text{and}\qquad\|u\|_{\dot{H}^s_0}:=\|(I-\Delta_D)^{s/2}u\|_{L^2}.
    \end{equation}
\end{definition}
\begin{remark}
    The spaces $H^1(\SG)$ and $H^1_0(\SG)$ coincide with the domain of the Neumann, respectively Dirichlet, energy forms $(\mathcal{E},\mathcal{F})$ and $(\mathcal{E},\mathcal{F}_0)$ from Section~\ref{SS:SG_basics}. Notice that in this case, $\mathcal{H}_0\subseteq H_N^s(\SG)$, see e.g.~\cite[Theorem 3.7]{Str03} and~\cite[Corollary 3.3]{CQ22}.
\end{remark}

The definition above can be generalized to higher orders of regularity, see e.g.~\cite{Str03,CQ22}. Worthwhile mentioning is the fact that both spaces coincide in the regularity range of interest for our purposes: It was proved in~\cite[Corollary 3.6]{Str03}, see also~\cite[Theorem 3.2]{CQ22} that $H^s(\SG)= H^s_0(\SG)$ when $0<s<\frac{d_S}{2}$. 

\medskip

Where does the latter regularity range come from? It is in fact dictated by the fractional Sobolev embedding theorem, which is a particular case of more general embeddings proved in~\cite[Theorem 3.11]{Str03} and~\cite[Theorem 4.1]{HZ05}.
\begin{theorem}\label{T:Sobolev_embedding}
Let $0<s<2$ and $1<q\leq\infty$. Then, $H^{s}(\SG)\subseteq L^q(\SG)$ holds 
\begin{enumerate}[wide=0em,leftmargin=1.75em,label={\rm (\roman*)}]
    \item for $0<s<\frac{d_S}{2}$ and $q=\frac{2d_S}{d_S-2s}$,
    \item for $s=\frac{d_S}{2}$ and any $1<q<\infty$,
    \item for $s>\frac{d_S}{2}$ and $q=\infty$.
\end{enumerate}
The same embeddings hold for $H^s_0(\SG)$.
\end{theorem}
Once again we note that, on $\SG$, the spectral dimension $d_S$ satisfies $\frac{d_S}{2}<1<\frac{d_H}{2}$. This is in contrast with the classical setting: On the $d$-dimensional torus $\mathbb{T}^d$ or the sphere $\mathbb{S}^d$, it holds that $d_S=d_H=d$ and the threshold for the embedding is thus $\frac{d}{2}$. 
For convenience, we rephrase the relevant embedding from Theorem~\ref{T:Sobolev_embedding} in the following manner.

\begin{corollary}\label{C:Sobolev_embedding_b}
    Let $2<q< \infty$. Then $H^{\sigma_q}(\SG)\subseteq L^q(\SG)$, where
    \begin{equation*}%\label{E:def_sigma_q}
        \sigma_q:=\frac{d_S}{2}\bigg(1-\frac{2}{q}\bigg).
    \end{equation*}
\end{corollary}

Note that, in the case $q=\infty$, it only holds that $H^{\sigma_q+\varepsilon}(\SG)\subseteq L^q(\SG)$, however we will not make use of this embedding. As already mentioned, this note will focus on the low-regularity range, that is,
\begin{center}
    \emph{assume throughout that $0<s<\frac{d_S}{2}$}.
\end{center}

\medskip

In addition, and to avoid notational burden, we will omit ``$\SG$'' from all function spaces and only consider the Neumann Sobolev space $H^s$. 
 
%%%%%%%---------------------------------------------------
\subsection{The Schr\"odinger equation}
%%%%%%%---------------------------------------------------
The main purpose of this paper is to study the well-posedness of the nonlinear Schr\"odinger equation 
\begin{equation}\label{E:NLS_main}\tag{${\rm NLS}_k$}
    \begin{cases}
    i\partial_t u+\Delta u= |u|^{2k}u,&\quad (t,x)\in \mathbb{R}\times\SG\\
    u(0,\cdot)=u_0\in H^s,
    \end{cases}
\end{equation}
with $k\geq 1$, in the low-regularity range $0<s<\sigma_\infty$. To begin with, we shall consider the solution to its associated linear (free) equation
\begin{equation}\label{E:LS_main}\tag{LS}
    \begin{cases}
    i\partial_t u=-\Delta u,\quad(t,x)\in \mathbb{R}\times\SG\\
    u(0,\cdot)=u_0.
    \end{cases}
\end{equation}
Note first that the heat kernel estimates~\eqref{E:subG_HKE} in particular imply that the semigroup $\{P_t\}_{t\geq 0}$ is ultracontractive, see e.g.~\cite[Theorem 2.1.5]{Dav90}. The Laplace operator has thus compact resolvent and its spectrum is pure point with only accumulation point at infinity~\cite[Theorem 2.1.4]{Dav90}.
Moreover, we
\begin{center}
    \emph{choose once and for all an orthonormal basis of eigenfunctions} $\{\varphi\}$
\end{center}

\medskip

as described in detail in Section~\ref{SS:Eigenfunctions}. That basis will play the role of the ``Fourier basis'' in the sequel to write functions, operators and norms. In this way, the semigroup~\eqref{E:def_Pt} may be expressed as
\begin{equation}\label{E:Semigroup_spectral}
    P_tu(x)=e^{t\Delta}u(x)=\sum_{\varphi}e^{-\lambda_{\varphi}t}\langle u,\varphi\rangle\varphi(x),
\end{equation}
where $\lambda_\varphi$ denotes the eigenvalue associated with the eigenfunction $\varphi$. The fractional Sobolev seminorm~\eqref{E:def_seminorm_Hs} may as well be written as
\begin{equation}\label{E:spectral_seminorm_Hs}
    \|u\|_{\dot{H}^s}^2=\sum_{\varphi}(1+\lambda_\varphi)^s|\langle\varphi,u\rangle|^2.
\end{equation}
The representation~\eqref{E:Semigroup_spectral} is convenient to define the \emph{Schr\"odinger propagator} associated with the linear equation~\eqref{E:LS_main}.
\begin{proposition}%\label{P:solution_free}
    For any $u_0\in {\rm dom}\,(-\Delta)$, the function
    \begin{equation}\label{E:spectral_sol_LS}
        u(t,x)=S_tu_0(x):=\sum_{\varphi}e^{-it\lambda_{\varphi}}\langle u_0,\varphi\rangle\varphi(x),\quad(t,x)\in \mathbb{R}\times\SG
    \end{equation}
    is a solution to the linear Schr\"odinger equation~\eqref{E:LS_main}.
\end{proposition}
\begin{proof}
   Since $S_tu=P_{it}u_0$ and $\partial_s P_su_0=-\Delta P_su_0$, c.f.~\cite[A.1.5]{BGL14}, 
\[
    i\partial_t S_tu_0(x)=i\partial_t P_{it}u_0=(i)^2[\partial_s P_s]_{s=it} u_0=-[(-\Delta)P_s]_{s=it} u_0=\Delta P_{it} u_0=\Delta S_t u_0.
\]
\end{proof}
\begin{remark}\label{R:Isometry}
    A short explicit computation using the representation~\eqref{E:def_seminorm_Hs} shows that the operator $S_t$ is an isometry in $H^s$ and $H^s_0$ for any $t>0$.
\end{remark}

%%%%%%%---------------------------------------------------
%%%%%%%---------------------------------------------------
\section{Sobolev embedding}\label{S:Sobolev_embedding}

%%%%%%%---------------------------------------------------
\subsection{Eigenfunctions on the Sierpinski gasket}\label{SS:Eigenfunctions}
The characterization of the standard Dirichlet and Neumann Laplace operator on $\SG$ discussed in Section~\ref{SS:SG_basics} also provides a way to characterize their associated eigenfunctions: Neumann eigenfunctions $\varphi\in H^1$ satisfy
\begin{equation*}%\label{E:Neumann_efct_DF_char}
    \mathcal{E}(\varphi,v)=\lambda_{\varphi}\langle\varphi,v\rangle\qquad\forall\,v\in H^1,
\end{equation*}
where $\lambda_{\varphi}\geq 0$ is the associated eigenvalue, whereas Dirichlet eigenfunctions $\tilde{\varphi}\in H^1_0$ satisfy
\begin{equation*}%\label{E:Dirichlet_efct_DF_char}
    \mathcal{E}(\tilde{\varphi},v)=\lambda_{\tilde{\varphi}}\langle\tilde{\varphi},v\rangle\qquad\forall\,v\in H^1_0,
\end{equation*}
and in this case $\lambda_{\tilde{\varphi}}> 0$ is the associated eigenvalue, 
see e.g.~\cite[Proposition 4.1.2]{Kig01}. 

\medskip

Both the spectrum and a basis of eigenfunctions of $-\Delta$ can be obtained through the spectral decimation method introduced in~\cite{RT82}; specifically for eigenfunctions see~\cite[Section 4]{DSV99} or~\cite[Section 3.3]{Str06}. The method provides an algorithm that is applied recursively to extend eigenfunctions at each finite approximation level $V_m$ to the whole $\SG$. Without going into the details, we do point out that the basis obtained through this process is \emph{not} orthogonal, yet it can be orthonormalized using Gram-Schmidt, c.f.~\cite[Section 3]{ORS10}. What is relevant for our purposes is the fact that the basis will contain \emph{localized} eigenfunctions with compact support strictly inside $\SG$.

\medskip

The existence of localized eigenfunctions is an unprecedented feature of the Laplace operator on sufficiently symmetric fractals like $\SG$. Localized functions are simultaneously Neumann and Dirichlet eigenfunctions, and they vanish outside an open set not touching the boundary $V_0$. In~\cite{BK97}, Barlow and Kigami showed that the existence of such eigenfunctions in fact implies the possibility to construct localized eigenfunctions that are zero outside an open set $U\subset\SG$ \emph{for any} open set $U\subset\SG$. In particular, it is possible to find eigenfunctions with suitably large eigenvalues whose support has size comparable to an inverse power of their eigenvalue. Incidentally, these eigenvalues have very high multiplicity. %, which in turn has very high multiplicity.
The construction of such eigenfunctions can be read e.g. in~\cite[Section 3.4]{Str06}; we recall it in Proposition~\ref{P:existence_localized} below. 

\begin{proposition}\label{P:existence_localized}
    There exists a family of localized eigenfunctions $\{\psi^{(j)}\}_{j\geq 2}$ that satisfy
    \begin{enumerate}[wide=0em,itemsep=.5em,label={\rm (\roman*)}]
        \item $\displaystyle\|\psi^{(j)}\|_{L^2}=1$,
        \item\label{E:localized_evalue} $\displaystyle \lambda_{\psi^{(j)}}= c_65^{j}=:\Lambda_j$ {with $c_6>0$ independent of $j$},
       \item\label{E:localized_feature} $\displaystyle 3^{-j}\leq\frac{1}{2}\mu(\supp\psi^{(j)})\leq 3^{-j+1}$.
    \end{enumerate}
\end{proposition}

\begin{proof}
    Fix $j\geq 2$ and $p\in V_1{\setminus}V_0$. Further, consider the unique elements $w,w'\in \{0,1,2\}^{j-1}$ for which $p=F_w(\SG)\cap F_{w'}(\SG)$ and the function $v_j\colon V_k\to\mathbb{R}$ defined as follows: $v(p)=2$, $v(x)=-1$ for $x\in F_w(V_0)\cup F_{w'}(V_0)$ direct neighbor of $p$, $v(x)=1$ for $x\in F_w(V_0)\cup F_{w'}(V_0)$ common neighbor of $q,q'\in F_w(V_0)\cup F_{w'}(V_0)$ with $v(q)=v(q')=-1$, and zero otherwise; c.f. Figure~\ref{F:6series_Vj}.
    
    \begin{figure}[H]
        \includegraphics[scale=.35]{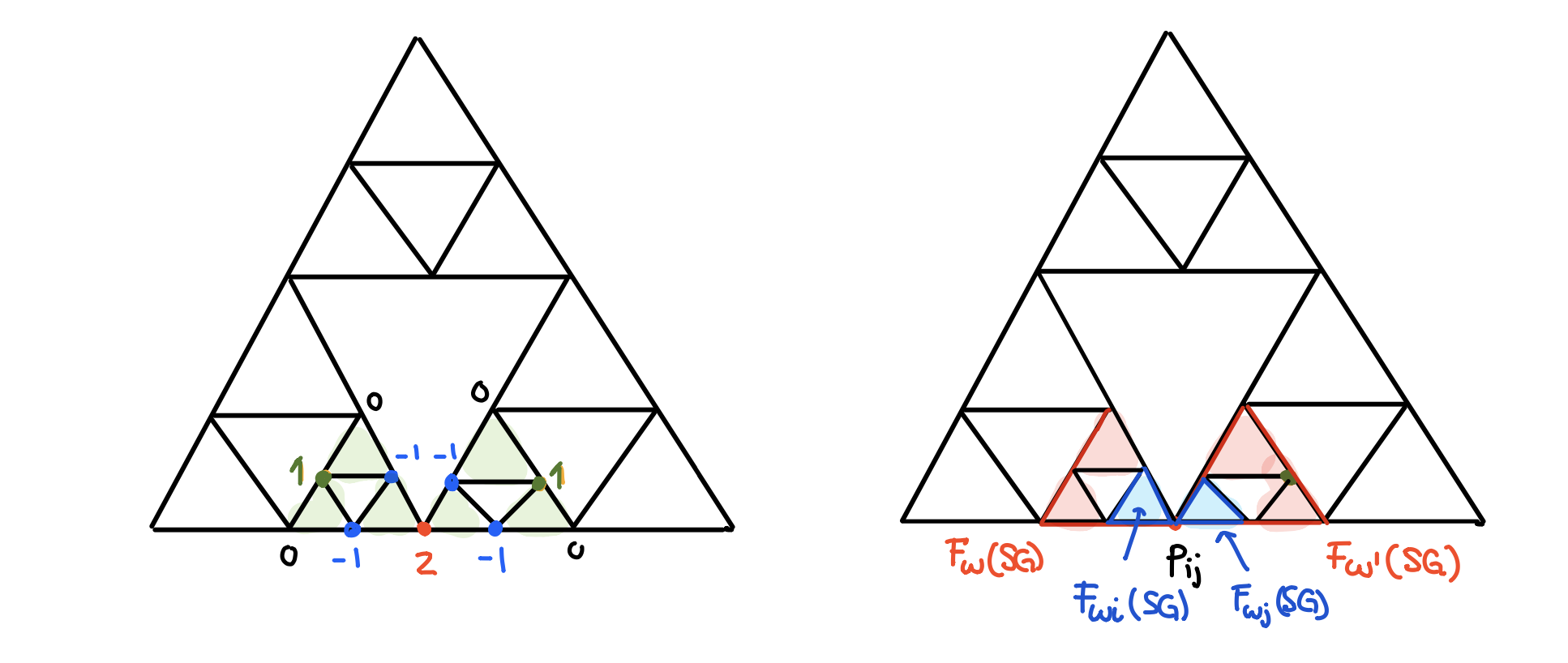}
        \caption{{\small A possible function $v\colon V_3\to\mathbb{R}$. In green the cells $F_w(V_0)$ and $F_{w'}(V_0)$. }}
        \label{F:6series_Vj}
    \end{figure}

    This is an eigenfunction of the graph Laplacian associated with the graph energy~\eqref{E:graph_energy_m} with both Dirichlet and Neumann boundary conditions~\cite[Section 3.4]{Str06}. Applying the spectral decimation method~\cite[Algorithm 2.4]{DSV99}, the function $v_k$ extends to an eigenfunction $\psi_j\in L^2$ with eigenvalue 
    \[
    \lambda_{\psi_k}=5^{j-2}\lambda_1^{(6)},
    \]
    where $\lambda_1^{(6)}\gg 1$ corresponds to the lowest non-zero Neumann eigenvalue and the sixth-lowest Dirichlet eigenvalue, c.f.~\cite[Theorem 5.1]{GRS01} or~\cite[Example 2.1]{AR22}. This proves (ii) with $c_6= 5^2\lambda_1^{(6)}$.
    
    Further, by construction it holds that $p=F_i(p_k)$ for some $0\leq i<k\leq 2$ and thus
    \[
    F_{wi}(V_0)\cup F_{w'k}(V_0)\subset \supp\psi_j\subset F_w(V_0)\cup F_{w'}(V_0),
    \]
    hence (iii) follows from~\eqref{E:def_measure_SG}. 
    Setting $\psi^{(j)}:=\frac{\psi_j}{\|\psi_j\|_{L^2}}$ we obtain the desired eigenfunction that also satisfies (i).
\end{proof}

\begin{remark}%\label{R:existence_localized}
    The eigenfunctions $\{\psi^{(j)}\}_{j\geq 2}$ from Proposition~\ref{P:existence_localized} belong to different eigenspaces, hence they are orthogonal to each other. In addition, since $\frac{d_S}{2}=\frac{\log 3}{\log 5}$, it follows that
    \begin{equation}\label{E:support_size}
        \mu(\supp\psi_j)\simeq 3^{-j}=(5^j)^{-d_S/2}=c_6^*\lambda_{\psi_j}^{-d_S/2}
    \end{equation}
   {with $c_6^*=(c_6)^{-d_S/2}$ and $c_6>0$ as in Proposition~\ref{P:existence_localized}\ref{E:localized_evalue}.}
\end{remark}
The eigenfunctions from Proposition~\ref{P:existence_localized} are constructed using the spectral decimation method and are orthogonal. In what follows, we choose as the orthonormal basis $\{\varphi\}$ of Neumann eigenfunctions of $L^2$ for~\eqref{E:def_Hs} the one obtained by: 
\begin{enumerate}[wide=0em,label=\arabic*)]
    \item performing the standard spectral decimation method, see e.g.~\cite[Section 3.3]{Str06};
    \item normalizing the eigenfunctions (note that the $\{\psi^{(j)}\}_{j\geq 2}$ from Proposition~\ref{P:existence_localized} are included);
    \item performing Gram-Schmidt while keeping $\{\psi^{(j)}\}_{j\geq 2}$ unchanged (which were already orthogonal).
\end{enumerate}
Since $\{\psi^{(j)}\}_{j\geq 2}$ are also Dirichlet eigenfunctions, a Dirichlet eigenbasis of $L^2$ that contains them can be constructed similarly.
%%%%%%%---------------------------------------------------
%%%%%%%---------------------------------------------------
\subsection{Sobolev embedding}
We will now use the orthonormal basis of Neumann, respectively Dirichlet, eigenvalues described in the previous section to express the fractional Sobolev seminorms~\eqref{E:spectral_seminorm_Hs}. 
The localized eigenfunctions $\{\psi^{(j)}\}_{j\geq 2}$ from Proposition~\ref{P:existence_localized} are examples of functions for which the Sobolev embedding from Theorem~\ref{T:Sobolev_embedding} is in a sense ``attained'', a phenomenon first observed in~\cite[Theorem 5.4.5]{Kig01}. Let the set of localized eigenfunctions from Proposition~\ref{P:existence_localized} be denoted by
\begin{equation}\label{E:def_Eloc}
    E_{\rm loc}:=\{\psi^{(j)}\}_{~j\geq 2}=\{\text{localized eigenfunctions}\}.
\end{equation} 

\begin{theorem}\label{T:Sobolev_saturates}
    Let $q=2k>2$ with $k\geq 1$. Any $\psi\in E_{\rm loc}$ satisfies
    \begin{equation*}%\label{E:Sobolev_saturates}
        \|\psi\|_{L^q}\simeq_{d_S,q} \|\psi\|_{{H}^{\sigma_q}}\simeq \lambda_{\psi}^{\sigma_q/2},\qquad\sigma_q=\frac{d_S}{2}\Big(1-\frac{2}{q}\Big),
    \end{equation*}
    with constants independent of the function $\psi$. The same holds for the norm $\|\cdot\|_{{H}^s_0}$.
\end{theorem}

\begin{proof}
    By virtue of H\"older's inequality and Proposition~\ref{P:existence_localized}~\ref{E:localized_feature}, see also~\eqref{E:support_size}, it follows that
    \begin{align*}
        1&=\|\psi\|_{L^2}^2=\int_{\supp\psi}\psi^2d\mu
        \leq \bigg(\int_{\supp\psi}\mathbf{1}^pd\mu\bigg)^{1/p}\bigg(\int_{\SG}\psi^{2p'}d\mu\bigg)^{1/p'}\\
        &=\mu(\supp\psi)^{1/p}\|\psi\|_{L^{2p'}}^2
        \leq 63^{-j} \|\psi\|_{L^{2p'}}^2=6c_{d_S,p}\lambda_{\psi}^{-\frac{d_S}{2p}}\|\psi\|_{L^{2p'}}^2,
    \end{align*}
    where $\frac{1}{p}+\frac{1}{p'}=1$. Thus, $\|\psi\|_{L^{2p'}}^2\apprge_{d_S,p} \lambda_{\psi}^{d_S/(2p)}$ holds for any $p>2$. Let now $q=2p'$. Then, $\frac{1}{p}=1-\frac{1}{p'}=1-\frac{2}{q}$ and therefore
    \begin{equation*}
        \|\psi\|_{L^{q}}^2\apprge_{d_S,q} \lambda_{\psi}^{\frac{d_S}{2}(1-\frac{2}{q})}\simeq_{d_S,q}\|\psi\|_{{H}^{\sigma_q}}^2.
    \end{equation*}
    In addition, for $2<q\leq \infty$, the Sobolev embedding from Theorem~\ref{T:Sobolev_embedding} yields 
    \[
        \|\psi\|_{L^{q}}^2\apprle_{d_S,q} \|\psi\|_{{H}^{\sigma_q}}
    \]
    c.f.\ Corollary~\ref{C:Sobolev_embedding_b}. The claim is now proved, and the same arguments with the Dirichlet eigenbasis yield the statement in terms of $\|\cdot\|_{H^s_0}$.
\end{proof}
%%%%%%%---------------------------------------------------
\begin{remark}
From now on, we will use $\sigma_\infty$ instead of $\frac{d_S}{2}$ to stress the role of the latter quantity as the threshold for the Sobolev embedding, meaning $H^s\subseteq L^\infty$ for all $s>\sigma_\infty$, c.f. Theorem~\ref{T:Sobolev_embedding}.
\end{remark}
%%%%%%%---------------------------------------------------
%%%%%%%---------------------------------------------------
\section{{Negative results}}\label{S:main_result}
This section is devoted to prove the main result{s} of the paper, Theorem~\ref{T:no_cont_below_Sobolev} {and Theorem~\ref{T:no_Strichartz_cubic}}. Specifically, we find initial data for which the flow map associated with~\eqref{E:NLS_main} fails to be $C^{2k+1}$-differentiable from $H^s$ to $H^s$ for any $0<s<\sigma_\infty$, {for which the standard Strichartz estimate for cubic NLS~\eqref{Stri} does not hold below regularity $\frac{d_S}{4}$}.

\medskip

{To prove Theorem~\ref{T:no_cont_below_Sobolev} we follow}~\cite[Section 6]{Bou97}, we perform the method of Taylor expansions in our current setup. To do so, consider for fixed $k\geq 1$, $T>0$, and $0<s<\sigma_\infty$, the initial value problem
\begin{equation}\label{E:NLS_main_param}
    \begin{cases}
    i\partial_t u+\Delta_{\SG} u- \mu|u|^{2k}u=0,&\quad (t,x)\in [-T,T]\times\SG,\\
    u(0,\cdot)=\gamma u_0,&\quad\gamma>0
    \end{cases}
\end{equation}
for $u_0\in H^s_x$ and $\mu=\pm$. Observe next that, if there was a solution to~\eqref{E:NLS_main} given by an iterative method, then the solution map would be smooth and in particular one could find $C_T>0$ such that for all $u_0\in H^s$
\begin{equation}\label{E:derivative_unif_bound}
    \Big\|\partial_\gamma^m u(\gamma,t,x){\big|_{\gamma=0}}\Big\|_{H^s}\leq C_T\|u_0\|_{H^s}^{m}\qquad \forall\;t\in[-T,T]
\end{equation}
for any $m\geq 1$. However, we will show in this section that~\eqref{E:derivative_unif_bound} fails with $m=2k+1$ for certain initial data. As a result, it will not be possible to obtain local well-posedness for~\eqref{E:NLS_main} by a fixed-point argument.

\subsection{Map derivatives}
We start by computing the derivatives involved in~\eqref{E:derivative_unif_bound} {at the zero initial condition}, see also~\cite[Section 3.8]{Tao06} for the particular case $k=1$. Note that the results are independent of the underlying space and thus hold anywhere the operators and spaces involved can be properly defined. 
Without loss of generality and for simplicity, we set $\mu=1$ in~\eqref{E:NLS_main_param}.

\begin{lemma}\label{L:map_derivatives}
Let $k\geq 1$, $T>0$ and $u(\gamma,t,x)$ be a solution to~\eqref{E:NLS_main_param} with $t\in[-T,T]$. For any $u_0\in H^s$ it holds that
\begin{equation*}
    \partial^m_\gamma u(\gamma,t,x){\big|_{\gamma=0}}
    =\begin{cases}
        0&\text{if}\quad m=0\\
        S_tu_0(x)&\text{if}\quad m=1,\\
        0&\text{if}\quad1<m<2k+1,\\
        \displaystyle (2k+1)!i\int_0^t S_{t-\tau}|S_\tau u_0(x)|^{2k}S_\tau u_0(x)\,d\tau&\text{if}\quad m=2k+1.
    \end{cases}
\end{equation*}
\end{lemma}
\begin{proof}
Note first that the solution to~\eqref{E:NLS_main_param} with zero initial data is
\begin{equation}\label{E:u_gamma_at_zero}
u(0,t,x)=0=\overline{u}(0,t,x)\qquad \forall\,t\in[-T,T]
\end{equation} 
and {hence}
\begin{multline*}
    \partial_\gamma\big[|u(\gamma,\tau,x)|^{2k}u(\gamma,\tau,x)\big]{\big|_{\gamma=0}}\\
    =\partial_\gamma\big[|u(\gamma,\tau,x)|^{2k}\big]_{\big|_{\gamma=0}}\!\!\!\!\!{u(0,t,x)}
    +\partial_\gamma [u(0,\tau,x)]_{\big|_{\gamma=0}}\!\!\!\!\!|u(0,\tau,x)|^{2k}=0.
\end{multline*}
{Therefore, } Duhamel's formula yields
\begin{equation*}
    u(\gamma,t,x)=\gamma S_t u_0(x)-i\int_0^tS_{t-\tau}\big(|u(\gamma,\tau,x)|^{2k}u(\gamma,\tau,x)\big)\,d\tau
\end{equation*}
{and thus}
\begin{equation}\label{E:partial_u_gamma}
    \begin{aligned}
    \partial_\gamma u(\gamma,t,x){\big|_{\gamma=0}}
    &=S_tu_0(x)-i\int_0^t S_{t-\tau}\big(\partial_\gamma\big[|u(\gamma,\tau,x)|^{2k}u(\gamma,\tau,x)\big]{\big|_{\gamma=0}}\big)\,d\tau\\
    &=S_tu_0(x).
    \end{aligned}
\end{equation}

Let now $1< m\leq 2k+1$ and set 
\begin{equation*}
    u_i:=u_i(\gamma,\tau,x)\begin{cases}
    u(\gamma,\tau,x)&\text{for}\quad 1\leq i\leq k+1,\\ 
    \overline{u}(\gamma,\tau,x)&\text{for}\quad k+2\leq i\leq 2k+1,
    \end{cases}
\end{equation*}
so that $|u|^{2k}u=\prod_{i=1}^{2k+1}u_i$. The product rule thus reads
\begin{multline*}
    \partial^m_\gamma(|u|^{2k}u)
    =\partial^m_\gamma(u^{k+1}\overline{u}^k)\\
    =\partial^m_\gamma\Big(\prod_{i=1}^{2k+1}u_i\Big)
    =\sum_{\alpha_1+\cdots+\alpha_{2k+1}=m}\binom{m}{\alpha_1,\ldots,\alpha_{2k+1}}\prod_{i=1}^{2k+1}\partial_\gamma^{\alpha_i}u_i,
\end{multline*}
where $\binom{m}{\alpha_1,\ldots,\alpha_{2k+1}}$ denotes the multinomial coefficient $\frac{m!}{\alpha_1!\alpha_2!\cdots\alpha_{2k+1}!}$.

\medskip

Observe now that whenever $1\leq m<2k+1$, then $\alpha_i=0$ for at least one $1\leq i\leq 2k+1$ in all terms of the summation above. Let $i_0$ denote that index. In view of~\eqref{E:u_gamma_at_zero} it follows that
\begin{equation*}
    \Big[\prod_{i=1}^{2k+1}\partial_\gamma^{\alpha_i}u_i\Big]{\big|_{\gamma=0}}=u_{i_0}(0,\tau,x)\prod_{i\neq i+0}^{2k+1}\partial_\gamma^{\alpha_i}{u_i}{\big|_{\gamma=0}}=0.
\end{equation*}
When $m=2k+1$, there is one non-zero term in the summation which corresponds to $\alpha_1=\ldots=\alpha_{2k+1}$. Combined with~\eqref{E:partial_u_gamma} we find that
\begin{equation*}
    \partial^m_\gamma(|u|^{2k}u)(\gamma,\tau,x)_{\big|_{\gamma=0}}=
    \begin{cases}
    0&\text{if}\quad 1<m<2k+1,\\
    (2k+1)!S_tu_0&\text{if}\quad m=2k+1
    \end{cases}
\end{equation*}
and conclude that
\begin{align*}
    \partial_\gamma^m u(\gamma,t,x){\big|_{\gamma=0}}
    &=\int_0^t S_{t-\tau}\big(\partial^m_\gamma\big[|u(\gamma,\tau,x)|^{2k}u(\gamma,\tau,x)\big]_{\big|_{\gamma=0}}\big)\,d\tau\\
    &=\begin{cases}
    0&\text{if}\quad 1<m<2k+1,\\
    \displaystyle (2k+1)!i\int_0^t S_{t-\tau}|S_\tau u_0(x)|^{2k}S_\tau u_0(x)\,d\tau&\text{if}\quad m=2k+1
    \end{cases}
    \end{align*}
as claimed.
\end{proof}

%%%-------------------------------------------------
%%%-------------------------------------------------
\subsection{\texorpdfstring{$C^{2k+1}$}{C2k}-ill-posedness}

We now show that~\eqref{E:derivative_unif_bound} fails for any initial data in the set of localized eigenfunctions constructed in Proposition~\ref{P:existence_localized}.
To begin with, we prove the following key estimate satisfied by any localized eigenfunction $\psi\in E_{\rm loc}$.
\begin{lemma}\label{L:Duhamel_lower_bound}
    Let $0<s<\frac{d_S}{2}$ and $k\geq 1$. For any $\psi\in E_{\rm loc}$ and $t>0$,
    \begin{equation*}%\label{E:Duhamel_fails}
        \bigg\|\int_0^tS_{t-\tau}|S_\tau\psi(x)|^{2k}S_\tau\psi(x)\,d\tau\bigg\|_{\dot{H}^s}
        \geq C_{d_S,k} |t|\|\psi\|_{\dot{H}^s}^{2k+1}\lambda_\psi^{k(\frac{d_s}{2}-s)}.
    \end{equation*}
\end{lemma}

\begin{proof}
    Let $\psi\in E_{\rm loc}$. In view of~\eqref{E:spectral_sol_LS}, $S_\tau\psi(x)=e^{it\lambda_{\psi}}\psi(x)$. Since $\psi$ is real-valued,
    \begin{equation}\label{E:Duhamel_fails_H01}
    \begin{aligned}
        \int_0^tS_{t-\tau}\big(|S_\tau\psi|^{2k}S_\tau\psi\big)\,d\tau
        &=\int_0^t\sum_{\varphi}e^{i\lambda_{\varphi}(t-\tau)}e^{i\tau \lambda_\psi}\langle\psi^{2k+1},\varphi\rangle\varphi(x)\,d\tau\\
        &=\sum_{\varphi}e^{it\lambda_{\varphi}}\langle\psi^{2k+1},\varphi\rangle\varphi(x)\int_0^te^{i\tau (\lambda_\psi-\lambda_\varphi)}\,d\tau\\
        &=\sum_{\varphi}\langle\psi^{2k+1},\varphi\rangle F(t,\lambda_\psi,\lambda_\varphi)\varphi(x),
    \end{aligned}
    \end{equation}
        where 
        \begin{equation}\label{E:Duhamel_fails_H02}
            F(t,\lambda_\psi,\lambda_\varphi)
            =\begin{cases}
                e^{it\lambda_{\varphi}}\Big[\frac{1}{i(\lambda_\psi-\lambda_\varphi)}e^{i\tau(\lambda_\psi-\lambda_\varphi)}\Big]_0^t=\frac{e^{it\lambda_\psi}-e^{it\lambda_\varphi}}{i(\lambda_\psi-\lambda_\varphi)}&\text{if }\lambda_\psi\neq\lambda_\varphi,\\
            te^{it\lambda_{\varphi}}&\text{if }\lambda_\psi=\lambda_\varphi.
            \end{cases}
        \end{equation}
    By Parseval, the latter implies that
    \begin{equation}\label{E:Duhamel_Hs_02}
        \bigg\|\int_0^tS_{t-\tau}\big(|S_\tau\psi|^{2k}S_\tau\psi\big)\,d\tau\bigg\|_{\dot{H}^s_x}^2
        =\sum_{\varphi}(1+\lambda_{\varphi})^s|\langle\psi^{2k+1},\varphi\rangle|^2|F(t,\lambda_\varphi,\lambda_\psi)|^2.
    \end{equation} 
   Since
   \[
   |F(t,\lambda_\varphi,\lambda_\psi)|^2
   =\begin{cases}
    \Big|\frac{e^{it\lambda_\psi}-e^{it\lambda_\varphi}}{i(\lambda_\psi-\lambda_\varphi)}\Big|^2&\text{if }\lambda_\psi\neq\lambda_\varphi,\\
    t^2&\text{if }\lambda_\psi=\lambda_\varphi,
    \end{cases}
   \]
   it follows from~\eqref{E:Duhamel_Hs_02} and Theorem~\ref{T:Sobolev_saturates} with $q=2(k+1)$ that
    \begin{align*}
        \bigg\|\int_0^tS_{t-\tau}\big(|S_\tau\psi|^{2k}S_\tau\psi\big)\,d\tau\bigg\|_{\dot{H}^s_x}^2
        &\geq (1+\lambda_{\psi})^s|\langle\psi^{2k+1},\psi\rangle|^2t^2\\
        &=t^2(1+\lambda_{\psi})^s\|\psi\|_{L^{2(k+1)}_x}^{4(k+1)}\\
        &\geq  C_{d_S,k} t^2\lambda_\psi^{s+\frac{d_S}{2}(1-\frac{1}{k+1})2(k+1)}
        = t^2\lambda_\psi^{s+kd_S}.
    \end{align*}
    Consequently, and since $\|\psi\|_{H^s}\simeq\lambda_{\psi}^{s/2}$, see e.g. Theorem~\ref{T:Sobolev_saturates},
    \begin{equation}\label{E:Duhamel_Hs_03}
    \bigg\|\int_0^tS_{t-\tau}\big(|S_\tau\psi|^{2k}S_\tau\psi\big)\,d\tau\bigg\|_{\dot{H}^s_x}\geq C_{d_S,k} |t|\lambda_\psi^{\frac{s}{2}+k\frac{d_S}{2}}
    \simeq C_{d_S,k} |t|\|\psi\|_{\dot{H}^s_x}^{2k+1}\lambda_\psi^{k(\frac{d_S}{2}-s)}
    \end{equation}
    as we wanted to prove. 
    \end{proof}

    \begin{proof}[Proof of Theorem~\ref{T:no_cont_below_Sobolev}]
    Let us now disprove~\eqref{E:derivative_unif_bound}. Given {any } $T>0$ {and any $C_T>0$,}  we may choose $j_T\geq 2$ large enough so that
    \[
    5^{j_T k(\frac{d_S}{2}-s)}c_6^{k(\frac{d_S}{2}-s)}\geq \frac{2C_T}{C_{d_S,k}|T|},
    \]
    where $C_T>0$ is the constant in~\eqref{E:derivative_unif_bound}, $c_6>0$ is the constant in Proposition~\ref{P:existence_localized}(ii), and $C_{d_S,k}$ is the constant in~\eqref{E:Duhamel_Hs_03}. By virtue of Proposition~\ref{P:existence_localized}\ref{E:localized_evalue}, we may now take {a localized eigenfunction} $\psi_T:=\psi^{(j_T)}\in E_{\rm loc}$, {c.f.~\eqref{E:def_Eloc},} so that 
    \[
    \lambda_{\psi_T}^{k(\frac{d_S}{2}-s)}=c_6^{k(\frac{d_S}{2}-s)} 5^{j_Tk(\frac{d_S}{2}-s)}\geq \max\Big\{\frac{2C_T}{C_{d_S,k}|T|},2\Big\}.
    \]
    Substituting $\psi=\psi_T$ and $t=T$ in~\eqref{E:Duhamel_Hs_03}, it follows from Lemma~\ref{L:map_derivatives} and Lemma~\ref{L:Duhamel_lower_bound} that
    \[
        \Big\|\partial_\gamma^m u(\gamma,t,x){\big|_{\gamma=0}}\Big\|_{H^s}\geq 2C_T\|u_0\|_{{H}^s}^{2k+1}
    \]
    for $u_0=\psi_T$. 
    Note that $\|u_0\|_{{H}^s}^{2k+1}$ and $\|u_0\|_{\dot{H}^s}^{2k+1}$ are equivalent because $\lambda_{\psi_T}>1$.
    \end{proof}
%%%%%%%--------------------------------------------------
%%%%%%%--------------------------------------------------
{\subsection{No Strichartz estimates except for the trivial one}\label{SS:No_Strichartz_cubic}
We dedicate this subsection to the proof of Theorem~\ref{T:no_Strichartz_cubic}. The typical Strichartz estimate to address well-posedness for the cubic NLS on a finite interval $[0,T]$ is 
\begin{equation}\label{E:Strichartz_cubic}
    \|S_\tau u\|_{L^4_tL^4_x}\apprle_T \|u\|_{H^s}
\end{equation}
for a certain exponent $s$. The trivial estimate using  Sobolev embedding  is
\begin{equation*}
 \|S_\tau u\|_{L^4_tL^4_x}\leq T^{1/4}\sup_{[0,T]} \|S_\tau u\|_{L^4_x}\leq T^{1/4}\sup_{[0,T]}\|S_\tau u\|_{H^s_x}
 =T^{1/4}\|u\|_{H^s_x},
\end{equation*}
where, using Corollary \ref{C:Sobolev_embedding_b}, $s>\frac{d_S}{4}$.
In this section, we show by a direct computation that~\eqref{E:Strichartz_cubic} fails on SG for any $s< \frac{d_S}{4}$, hence proving the theorem.
\medskip
Once again, the localized eigenfunctions are responsible for this phenomenon and its proof will use computations from Lemma~\ref{L:Duhamel_lower_bound} as well as Theorem~\ref{T:Sobolev_saturates}. 
We note that,  while the result is stated and proved only for the $L^4$ norm, which is associated to  the cubic NLS, the arguments generalize to any $L^{2k+2}$ norm associated to a $2k+1$ nonlinearity, we do not report the details of this computation here.
\begin{proof}
    Note first that, for any $\tau>0$ and $u,v\in L^2$ it holds that
    \begin{equation}\label{E:no_Strichartz_H01}
        \langle u, \overline{S_\tau v}\rangle=\langle S_{-\tau}u,v\rangle.
    \end{equation}
    Let now $\psi\in E_{\rm loc}$ and $T>0$. Applying~\eqref{E:no_Strichartz_H01} with $u=|S_t \psi|^2S_t\psi$, $v=S_T\psi$ and $\tau=t-T$ one writes
    \begin{align}
            \|S_t\psi\|_{L^4_tL^4_x([0,T]\times\SG)}^4
            &=\int_{\SG}\int_0^T|S_t\psi|^4dt\,d\mu
            =\int_{\SG}\int_0^T\big(|S_t\psi|^2S_t\psi\big)\overline{S_t\psi}\,dt\,d\mu\notag\\
            &=\int_0^T\langle |S_t\psi|^2S_t\psi,\overline{S_t\psi}\rangle\,dt
            =\int_0^T\langle |S_t\psi|^2S_t\psi,\overline{S_{t-T}S_T\psi}\rangle\,dt\notag\\
            &=\int_0^T\langle S_{T-t}\big( |S_t\psi|^2S_t\psi\big),S_T\psi\rangle\,dt\label{E:3rd_derivative}\\
            &=\int_{\SG}S_T\psi\int_0^TS_{T-t}\big(|S_t\psi|^2S_t\psi\big)\,dt\,d\mu.\notag
    \end{align}
    In view of~\eqref{E:Duhamel_fails_H01} and~\eqref{E:Duhamel_fails_H02}, it follows from the latter expression that 
    \begin{equation*}
        \begin{aligned}
        \|S_t\psi\|_{L^4_tL^4_x([0,T]\times\SG)}^4
        &=\sum_{\varphi}\langle\psi^3,\varphi\rangle F(T,\lambda_\psi,\lambda_\varphi)\langle \varphi,S_T\psi\rangle\\ 
        &=\langle\psi^3,\psi\rangle F(T,\lambda_\psi,\lambda_\psi)e^{-iT\lambda_\psi}\\
        &=\|\psi\|_{L^4}^4T e^{iT\lambda_\psi}e^{-iT\lambda_\psi}=T\|\psi\|_{L^4}^4.
        \end{aligned}
    \end{equation*}
    Finally, note that $\sigma_q$ from Theorem~\ref{T:Sobolev_saturates} with $q=4$ equals $\frac{d_S}{2}(1-\frac{1}{2})=\frac{d_S}{4}$ and therefore~\eqref{E:no_Strichartz_cubic} follows.
\end{proof}}
\subsection{Connecting both results}
In the previous two sections we have shown directly and independently the non-existence of a $C^{3}$ flow map and the absence of Strichartz estimates for cubic NLS. In this section we establish the connection between these two results. Once again and for simplicity we only discuss the cubic case and leave the generic case to the interested reader.

\medskip

The result follows the lines of~\cite[Proposition 10]{GG10} along with a suitable replacement for dyadic intervals that is based on~\cite{Str05}, see also~\cite[Section 2.3]{AR22}.

\begin{proposition}\label{P:Prop10_GG10}
    If the cubic NLS
    \begin{equation}\label{E:NLS_SG}
        \begin{cases}
        i\partial_t u+\Delta_{\SG} u= |u|^{2k}u&\quad (t,x)\in \mathbb{R}\times\SG,\\
        u(0,\cdot)=u_0\in H^s,&
        \end{cases}
    \end{equation}
    is strongly well-posed on $H^{s}$ for some $s\geq 0$, then
    \begin{equation*}
    \|S_t u\|_{L^4_tL^4_x}\apprle_T \|u\|_{H^{r/2}}
    \end{equation*}
    for $t\in[0,T]$ and any $r> s$.
\end{proposition}
\begin{proof}
    Let $u_0\in H^{s}$ and $[0,T]$ be the interval in which~\eqref{E:NLS_SG} is strongly well-posed. First, note that the equality obtained in~\eqref{E:3rd_derivative} together with the expression of the third map derivative in Lemma~\ref{L:map_derivatives} imply
    \begin{equation*}
        \begin{aligned}
        \|S_tu_0\|_{L^4_tL^4_x([0,T]\times\SG)}^4
        &=\int_0^T\langle S_{T-t}\big( |S_tu_0|^2S_t\psi\big),S_Tu_0\rangle\,dt\\
        &= \langle -\frac{1}{6i}\partial_\gamma^3u(\gamma,t,x){\big|_{\gamma=0}},S_Tu_0\rangle\\
        &=\|S_Tu_0\|_{H^{-s}}\langle -\frac{1}{6i}\partial_\gamma^3u(\gamma,t,x){\big|_{\gamma=0}},\frac{S_Tu_0}{\|S_Tu_0\|_{H^{-s}}}\rangle\\
        &\leq \|S_Tu_0\|_{H^{-s}}\Big\|\partial_\gamma^3u(\gamma,t,x){\big|_{\gamma=0}}\Big\|_{H^{s}}\\
        &=\|u_0\|_{H^{-s}}\Big\|\partial_\gamma^3u(\gamma,t,x){\big|_{\gamma=0}}\Big\|_{H^{s}},
        \end{aligned}
    \end{equation*}
    where in the last line we use the fact that $S_T$ is an isometry in $H^{-s}$, c.f. Remark~\ref{R:Isometry}. 
    Thus, the strong regularity assumption yields the bound
    \begin{equation}\label{E:L4_vs_Hs_if_C3}
        \|S_tu_0\|_{L^4_tL^4_x([0,T]\times\SG)}^4\apprle_T \|u_0\|_{H^{-s}}\|u_0\|_{H^{s}}^3,
    \end{equation}
    see~\eqref{E:derivative_unif_bound}. 

    \medskip

    Let us now assume that $u_0$ belongs to a single eigenspace $E_\lambda$ for some eigenvalue $\lambda>0$. Then, in view of~\eqref{E:def_seminorm_Hs}, 
    \begin{equation}\label{E:Hs_eigenspace}
        \|u_0\|_{H^{s}}^2\simeq \sum_{k\geq 1}(1+\lambda_k)^{s}|\langle u_0,\varphi_k\rangle|^2=(1+\lambda)^{s}\|u_0\|_{L^2}^2
    \end{equation}
    and analogously $\|u_0\|_{H^{-s}}^2\simeq (1+\lambda)^{-s}\|u_0\|_{L^2}$. Therefore, it follows from~\eqref{E:L4_vs_Hs_if_C3} that
    \begin{equation}\label{E:L4_vs_Hs_if_C3_eigensp}
        \begin{aligned}
        \|S_tu_0\|_{L^4_tL^4_x([0,T]\times\SG)}
        &\apprle_T (1+\lambda)^{\frac{3s}{8}}\|u_0\|_{L^2}^{1/2}(1+\lambda)^{-\frac{s}{8}}\|u_0\|_{L^2}^{1/2}\\
        &=(1+\lambda)^{\frac{s}{4}}\|u_0\|_{L^2}\simeq_T \|u_0\|_{H^{s/2}} 
        \end{aligned}
    \end{equation}
    and the same estimate follows with $\|u_0\|_{H^{r/2}}$ on the right-hand side for any $r>s$.

    \medskip

    Finally, suppose that $u_0$ is a generic function in $H^{s}$. This function can be decomposed into ``dyadic blocks'' as follows: For each $j\geq 1$, let $N_j:=\frac{1}{2}(3^{j+1}-3)$ and 
    \begin{equation*}
    \pi_ju:=\sum_{k=N_{j-1}+1}^{N_j}\langle u,\varphi_k\rangle\varphi_k.
    \end{equation*}
    It is known, see e.g.~\cite{Str05} or~\cite[Section 2.2]{AR22}, that $\lambda_k\simeq \Lambda_j$ for all $N_{j-1}<k\leq N_j$ with $\Lambda_j$ from Proposition~\ref{P:existence_localized}(ii). Hence, as in~\eqref{E:Hs_eigenspace} and by orthogonality we obtain
    \begin{equation*}
        \|\pi_ju\|_{H^{s}}^2\simeq \Lambda_j^{s}\|\pi_j u\|_{L^2}^2\qquad\text{and}\qquad \|u\|_{H^{s}}^2\simeq \sum_{j=1}^\infty \|\pi_ju\|_{H^{s}}^2
    \end{equation*}
    for any $u\in H^{s}$. 
    Applying the latter and~\eqref{E:L4_vs_Hs_if_C3_eigensp} to the initial data $u_0\in H^{s}$ yields
    \begin{equation*}
        \begin{aligned}
            \|S_tu_0\|_{L^4_tL^4_x([0,T]\times\SG)}
            &=\Big\|\sum_{j=1}^\infty S_t\pi_ju_0\Big\|_{L^4_tL^4_x([0,T]\times\SG)}\\
            &\leq \sum_{j=1}^\infty \|S_t\pi_ju_0\|_{L^4_tL^4_x([0,T]\times\SG)}\\
            &\apprle_T \sum_{j=1}^\infty \Lambda_j^{\frac{s}{4}}\|\pi_j u\|_{L^2}\\
            &=\sum_{j=1}^\infty \Lambda_j^{-\varepsilon}\Lambda_j^{\frac{s}{4}+\varepsilon}\|\pi_j u\|_{L^2}\\
            &\leq \Big(\sum_{j=1}^\infty \Lambda_j^{-\varepsilon}\Big)^{1/2}\Big(\sum_{j=1}^\infty \Lambda_j^{\frac{s}{2}+2\varepsilon}\|\pi_j u\|_{L^2}^2\Big)^{1/2}\\
            &\apprle_T \|u\|_{H^{(s/2)+\varepsilon}}
        \end{aligned}
    \end{equation*}
    for an arbitrary $\varepsilon>0$ as desired.
\end{proof}
%%%%%%%---------------------------------------------------

\bibliographystyle{amsplain}
\bibliography{NLS_references}

\providecommand{\bysame}{\leavevmode\hbox to3em{\hrulefill}\thinspace}
\providecommand{\MR}{\relax\ifhmode\unskip\space\fi MR }
\providecommand{\MRhref}[2]{%
  \href{http://www.ams.org/mathscinet-getitem?mr=#1}{#2}
}
\providecommand{\href}[2]{#2}
\begin{thebibliography}{10}

\bibitem{AR22}
P.~Alonso~Ruiz, \emph{Minimal gap in the spectrum of the {S}ierpi\'nski
  gasket}, Int. Math. Res. Not. IMRN (2022), no.~23, 18874--18894.

\bibitem{BGL14}
D.~Bakry, I.~Gentil, and M.~Ledoux, \emph{Analysis and geometry of {M}arkov
  diffusion operators}, Grundlehren der mathematischen Wissenschaften
  [Fundamental Principles of Mathematical Sciences], vol. 348, Springer, Cham,
  2014.

\bibitem{Ban04}
V.~Banica, \emph{On the nonlinear {S}chr\"odinger dynamics on
  {$\mathbb{S^2}$}}, J. Math. Pures Appl. (9) \textbf{83} (2004), no.~1,
  77--98.

\bibitem{Bar98}
M.~T. Barlow, \emph{Diffusions on fractals}, Lectures on probability theory and
  statistics ({S}aint-{F}lour, 1995), Lecture Notes in Math., vol. 1690,
  Springer, Berlin, 1998, pp.~1--121.

\bibitem{BK97}
M.~T. Barlow and J.~Kigami, \emph{Localized eigenfunctions of the {L}aplacian
  on p.c.f.\ self-similar sets}, J. London Math. Soc. (2) \textbf{56} (1997),
  no.~2, 320--332.

\bibitem{BST99}
O.~Ben-Bassat, R.~S. Strichartz, and A.~Teplyaev, \emph{What is not in the
  domain of the {L}aplacian on {S}ierpinski gasket type fractals}, J. Funct.
  Anal. \textbf{166} (1999), no.~2, 197--217.

\bibitem{Bou93}
J.~Bourgain, \emph{Fourier transform restriction phenomena for certain lattice
  subsets and applications to nonlinear evolution equations. {I}.
  {S}chr\"odinger equations}, Geom. Funct. Anal. \textbf{3} (1993), no.~2,
  107--156.

\bibitem{Bou97}
\bysame, \emph{Periodic {K}orteweg de {V}ries equation with measures as initial
  data}, Selecta Math. (N.S.) \textbf{3} (1997), no.~2, 115--159.

\bibitem{Bou99}
\bysame, \emph{Nonlinear {S}chr\"odinger equations}, Hyperbolic equations and
  frequency interactions (L.~Caffarelli and E.~Weinan, eds.), IAS/Park City
  Mathematics Series, American Mathematical Society, 1999, pp.~3--158.

\bibitem{B+24}
N.~Burq, N.~Camps, M.~Latocca, C.~Sun, and N.~Tzvetkov, \emph{The second
  {P}icard iteration of {NLS} on the 2d sphere does not regularize {G}aussian
  random initial data}, arXiv:2004.18241 (2024).

\bibitem{BGT02}
N.~Burq, P.~G\'erard, and N.~Tzvetkov, \emph{An instability property of the
  nonlinear {S}chr\"odinger equation on {$S^d$}}, Math. Res. Lett. \textbf{9}
  (2002), no.~2-3, 323--335.

\bibitem{BGT04}
\bysame, \emph{Strichartz inequalities and the nonlinear {S}chr\"odinger
  equation on compact manifolds}, Amer. J. Math. \textbf{126} (2004), no.~3,
  569--605.

\bibitem{BGT05}
\bysame, \emph{Bilinear eigenfunction estimates and the nonlinear
  {S}chr\"odinger equation on surfaces}, Invent. Math. \textbf{159} (2005),
  no.~1, 187--223.

\bibitem{CQ22}
S.~Cao and H.~Qiu, \emph{Sobolev spaces on p.c.f. self-similar sets {I}:
  {C}ritical orders and atomic decompositions}, J. Funct. Anal. \textbf{282}
  (2022), no.~4, Paper No. 109331, 47.

\bibitem{CCT03}
M.~Christ, J.~Colliander, and T.~Tao, \emph{Asymptotics, frequency modulation,
  and low regularity ill-posedness for canonical defocusing equations}, Amer.
  J. Math. \textbf{125} (2003), no.~6, 1235--1293.

\bibitem{CCT03b}
\bysame, \emph{Instability of the periodic nonlinear {S}chr\"odinger equation},
  arXiv:math/0311227 (2003).

\bibitem{DSV99}
K.~Dalrymple, R.~S. Strichartz, and J.~P. Vinson, \emph{Fractal differential
  equations on the {S}ierpinski gasket}, J. Fourier Anal. Appl. \textbf{5}
  (1999), no.~2-3, 203--284.

\bibitem{Dav90}
E.~B. Davies, \emph{Heat kernels and spectral theory}, Cambridge Tracts in
  Mathematics, vol.~92, Cambridge University Press, Cambridge, 1990.

\bibitem{FOT11}
M.~Fukushima, Y.~Oshima, and M.~Takeda, \emph{Dirichlet forms and symmetric
  {M}arkov processes}, extended ed., De Gruyter Studies in Mathematics,
  vol.~19, Walter de Gruyter \& Co., Berlin, 2011.

\bibitem{FS92}
M.~Fukushima and T.~Shima, \emph{On a spectral analysis for the
  {S}ierpi\'{n}ski gasket}, Potential Anal. \textbf{1} (1992), no.~1, 1--35.

\bibitem{GG10}
P.~G\'erard and S.~Grellier, \emph{The cubic {S}zeg{o}{} equation}, Ann. Sci.
  \'Ec. Norm. Sup\'er. (4) \textbf{43} (2010), no.~5, 761--810.

\bibitem{GRS01}
M.~Gibbons, A.~Raj, and R.~S. Strichartz, \emph{The finite element method on
  the {S}ierpinski gasket}, Constr. Approx. \textbf{17} (2001), no.~4,
  561--588.

\bibitem{HK24}
S.~Herr and B.~Kwak, \emph{Strichartz estimates and global well-posedness of
  the cubic {NLS} on {$\mathbb{T}^2$}}, Forum Math. Pi \textbf{12} (2024),
  Paper No. e14, 21.

\bibitem{HZ05}
J.~Hu and M.~Z\"{a}hle, \emph{Potential spaces on fractals}, Studia Math.
  \textbf{170} (2005), no.~3, 259--281.

\bibitem{Hut81}
J.~E. Hutchinson, \emph{Fractals and self-similarity}, Indiana Univ. Math. J.
  \textbf{30} (1981), no.~5, 713--747.

\bibitem{Kig01}
J.~Kigami, \emph{Analysis on fractals}, Cambridge Tracts in Mathematics, vol.
  143, Cambridge University Press, Cambridge, 2001.

\bibitem{Kum93}
T.~Kumagai, \emph{Estimates of transition densities for {B}rownian motion on
  nested fractals}, Probab. Theory Related Fields \textbf{96} (1993), no.~2,
  205--224.

\bibitem{KK24}
B.~Kwak and S.~Kwon, \emph{Critical local well-posedness of the nonlinear
  {S}chr\"odinger equation on the torus}, arXiv:2411.17147 (2024).

\bibitem{Mol09}
L.~Molinet, \emph{On ill-posedness for the one-dimensional periodic cubic
  {S}chrodinger equation}, Math. Res. Lett. \textbf{16} (2009), no.~1,
  111--120.

\bibitem{ORS10}
K.~A. Okoudjou, L.~G. Rogers, and R.~S. Strichartz, \emph{Szeg\"o{} limit
  theorems on the {S}ierpi\'nski gasket}, J. Fourier Anal. Appl. \textbf{16}
  (2010), no.~3, 434--447.

\bibitem{Ram84}
R.~Rammal, \emph{Spectrum of harmonic excitations on fractals}, J. Physique
  \textbf{45} (1984), no.~2, 191--206.

\bibitem{RT82}
R.~Rammal and G.~Toulouse, \emph{Random walks on fractal structures and
  percolation clusters}, J. Physique Lett. \textbf{43} (1982), L13--L22.

\bibitem{Str03}
R.~S. Strichartz, \emph{Function spaces on fractals}, J. Funct. Anal.
  \textbf{198} (2003), no.~1, 43--83.

\bibitem{Str05}
\bysame, \emph{Laplacians on fractals with spectral gaps have nicer {F}ourier
  series}, Math. Res. Lett. \textbf{12} (2005), no.~2-3, 269--274.

\bibitem{Str06}
\bysame, \emph{Differential equations on fractals}, Princeton University Press,
  Princeton, NJ, 2006, A tutorial.

\bibitem{Tao06}
T.~Tao, \emph{Nonlinear dispersive equations}, CBMS Regional Conference Series
  in Mathematics, vol. 106, Conference Board of the Mathematical Sciences,
  Washington, DC; by the American Mathematical Society, Providence, RI, 2006,
  Local and global analysis.

\bibitem{Yos80}
K.~Yosida, \emph{Functional analysis}, sixth ed., Grundlehren der
  Mathematischen Wissenschaften, vol. 123, Springer-Verlag, Berlin-New York,
  1980.

\end{thebibliography}
\end{document}